\documentclass[11pt,reqno]{amsart}
\usepackage{amsthm,amssymb,url,hyperref}
\usepackage{fullpage}
\usepackage{graphics}
\usepackage{graphicx}

\theoremstyle{plain}
\newtheorem{theorem}{Theorem}[section]

\newtheorem{lemma}[theorem]{Lemma}

\newtheorem{proposition}[theorem]{Proposition}
\newtheorem{observation}[theorem]{Observation}

\theoremstyle{definition}
\newtheorem{definition}[theorem]{Definition}
\newtheorem{problem}{Problem}

\newtheorem{example}[theorem]{Example}

\numberwithin{equation}{section}

\def\ldiv{\backslash}
\def\rdiv{/}

\def\cl{\mathrm{cl}}

\def\mlt#1{\mathrm{Mlt}(#1)}

\def\inn#1{\mathrm{Inn}(#1)}

\def\tup#1{\boldsymbol{#1}}

\title{Three concepts of nilpotence in loops}

\author{\v Zaneta Semani\v sinov\'a}
\address{(\v Z.S.) Institute of Algebra, Faculty of Mathematics, Technische Universit\"{a}t Dresden, Germany}
\email{zaneta.semanisinova@tu-dresden.de}

\author{David Stanovsk\'y}
\address{(D.S.) Department of Algebra, Faculty of Mathematics and Physics, Charles University, Prague, Czechia}
\email{stanovsk@karlin.mff.cuni.cz}

\thanks{D. Stanovsk\'y partially supported by the GA\v CR grant 18-20123S and by the cooperation grant LTAUSA19070. \v Z. Semani\v sinov\'a partially supported by the DFG grant 467967530. }



\begin{document}

\begin{abstract}
We introduce the abstract concept of supernilpotence in loop theory, and relate it to existing concepts, namely, central nilpotence and nilpotence of the multiplication group. We prove that the class of supernilpotence is greater or equal than the class of nilpotence of the multiplication group, and combining existing results, we show that a finite loop is supernilpotent if and only if its multiplication group is nilpotent. We also provide a new exposition of a classical result and crucial ingredient, that loops with a nilpotent multiplication group are centrally nilpotent and admit a prime decomposition.
\end{abstract}

\maketitle

\section{Introduction}

Loops generalize groups by dropping the axiom of associativity \cite{Br-book,Pfl}, so, naturally, many concepts in early loop theory were developed in direct analogy to group theory. For example, the center of a loop is defined as the set of all elements that commute \emph{and associate} with every other element, and subsequently we obtain the concept of central series and of \emph{central nilpotence}. 

Groups of prime power order are nilpotent and finite nilpotent groups admit a prime decomposition, i.e., they decompose as a direct product of groups of prime power order. Both statements generalize to Moufang loops \cite{G2,GW} (the proofs are much harder!), but neither statement holds in general: every non-associative loop of prime order is not centrally nilpotent (since the order of the center divides the order of the loop) and there is a directly indecomposable nilpotent loop of order~6.

The path towards a loop theoretic concept of nilpotence is certainly not unique. For example, one could consider \emph{nilpotence of the multiplication group}. 
Bruck \cite{Br} proved that every loop with a nilpotent multiplication group is centrally nilpotent, but the converse fails. Wright \cite{Wr} proved that a finite loop with a nilpotent multiplication group admits a prime decomposition, i.e., it is a direct product of nilpotent loops of prime power order.

In late 1970s, universal algebra established an abstract concept of the commutator of congruences, and subsequently of nilpotence and solvability of general algebraic structures \cite{FM,HH,Sm}. Applied to loop theory, it is not difficult to verify that the universal algebraic concept of nilpotence coincides with central nilpotence, although the two notions solvability are quite different \cite{SV1,SV2}. 

Over the last decade, universal algebra introduced a stronger concept of nilpotence, called \emph{supernilpotence} \cite{AM,Bu}, which, under certain assumptions met by groups and loops, implies prime decomposition. 
The purpose of the present note is to introduce the concept of supernilpotence into loop theory and to make initial observations on the relationship of the three concepts of nilpotence. 

Formal definitions will be presented in subsequent sections. To state the main theorem, let us introduce the following notation: for a loop $Q$, let
\begin{itemize}
	\item $\cl_{cn}(Q)$ denote the class of central nilpotence of $Q$,
	\item $\cl_{m}(Q)$ denote the class of nilpotence of $\mlt Q$,
	\item $\cl_{sn}(Q)$ denote the class of supernilpotence of $Q$.
\end{itemize}
If $Q$ fails to be centrally nilpotent, we will denote $\cl_{cn}(Q)=\infty$, \emph{et cetera} for the other properties.

\begin{theorem}\label{thm:main}
Let $Q$ be a loop. Then \[ \cl_{sn}(Q)\geq\cl_{m}(Q)\geq\cl_{cn}(Q). \]
Moreover, if $Q$ finite, then $\cl_{sn}(Q)<\infty$ if and only if $\cl_{m}(Q)<\infty$.
\end{theorem}


The first inequality is a new result, the second inequality is due to Bruck \cite[p. 282, Corollary III to Theorem 8B]{Br} (he states it for finite loops,  however, this assumption is never used in the proof), and the last part is a combination of a loop theoretic result of Wright \cite[Theorem 1]{Wr} and universal algebraic results based on works of Kearnes \cite{Kea} and Aichinger and Mudrinski \cite[Section 7]{AM}. It is an open problem whether the assumption of finiteness is necessary.


For a group $G$, $\cl_{sn}(G)=\cl_{cn}(G)$, and thus the three classes of nilpotence coincide \cite{AE,SV-sn}. 
For loops, the classes may be different: there is a loop $Q$ such that
\begin{itemize}
	\item $|Q|=6$, $\cl_{cn}(Q)=2$, $\cl_m(Q)=\infty$ (Example \ref{ex:6}). 
	\item $|Q|=8$, $\cl_{cn}(Q)=2$, $\cl_m(Q)=3$, $\cl_{sn}(Q)\geq4$ (Example \ref{ex:8}).
\end{itemize}

The structure of the paper is as follows.
In Section \ref{sec:ua}, we introduce the abstract concepts of nilpotence and supernilpotence.
In Section \ref{sec:loops}, we explain how these concepts are realized in loops.
In Section \ref{sec:proofs}, we prove Theorem \ref{thm:main}. Although two of the three parts of the theorem are not new, we present a complete proof, since parts of the original proofs are scattered around several old papers written in a somewhat outdated style, and they are omitted in recent surveys such as \cite{NR,Pfl}.

\section{(Super)nilpotence in universal algebra}\label{sec:ua}

For a general background in universal algebra, we refer to the textbook \cite{Berg}. We will focus on the the abstract notions of nilpotence and supernilpotence.

By a \emph{polynomial operation} on an algebraic structure (shortly, \emph{algebra}) $A$ we mean a term operation of the algebra $A$ enhanced with constants for every element of $A$. If no confusion arises, we will just use the noun \emph{polynomial}.

We are mostly interested in algebras with a \emph{Mal'tsev term}, i.e., a term $m$ satisfying the identities $m(x,y,y)=m(y,y,x)=x$. Groups and loops always have a Mal'tsev term, $m(x,y,z)=xy^{-1}z$ and $m(x,y,z)=(x/y)z$, respectively. Algebras with a Mal'tsev term will be called shortly \emph{Mal'tsev algebras}.

\subsection{Nilpotence}

The abstract commutator theory for congruence modular varieties was developed by Freese and McKenzie \cite{FM}, building upon earlier works of Smith \cite{Sm} and Hagemann and Herrmann \cite{HH}. In particular, they defined the center of an algebra, and subsequently, they define the abstract notion of nilpotence.

\begin{definition}
The \emph{center} of an algebra $A$ is the largest congruence $\zeta_A$ such that 
for every polynomial operation $p$, all pairs $a\,\zeta_A\,b$ and all tuples $\tup u,\tup v$ of elements of $A$
\begin{align*}
p(a,\tup u) &= p(a,\tup v) \\
&\Downarrow \\
p(b,\tup u) &= p(b,\tup v) 
\end{align*}
An algebra $A$ is called \emph{$k$-nilpotent} if it possesses a central series of congruences of length $k$, i.e., a series of congruences $0_A=\alpha_0\leq\alpha_1\leq\ldots\leq\alpha_k=1_A$ such that $\alpha_{i+1}/\alpha_i\leq\zeta_{A/\alpha_i}$.
The \emph{class of nilpotence}, $\cl_{n}(A)$, is the smallest $k$ such that $A$ is $k$-nilpotent.

(Equivalently, we could have defined the commutator of two congruences and subsequently, the lower and upper central series.)
\end{definition}

It is not difficult to prove that the center of a group in the present sense is the congruence corresponding to the standard center. Consequently, the two concepts of nilpotence are equivalent for groups.

\subsection{Supernilpotence}

The theory of higher commutators was introduced by Bulatov \cite{Bu} and developed mainly by Aichinger and Mudrinski \cite{AM} for Mal'tsev varieties and by Moorhead 
\cite{Moo} for congruence modular varieties. The outcome of the theory is a new, stronger notion of nilpotence, called \emph{supernilpotence}. Under mild universal algebraic assumptions (existence of a Taylor term), $k$-supernilpotence implies $k$-nilpotence \cite{Moo}. 

\begin{definition}\label{d:supnilp}
An algebra $A$ is called \emph{$k$-supernilpotent} if for every polynomial operation $p$ and all pairs of tuples $\tup a_1\neq\tup b_1,\dots,\tup a_k\neq\tup b_k$ and $\tup u,\tup v$ of elements of $A$ 
\begin{align*}
p(\tup x_1,\ldots,\tup x_k,\tup u) &= p(\tup x_1,\ldots,\tup x_k,\tup v) \qquad \forall (\tup x_1,\ldots,\tup x_k)\in\{\tup a_1,\tup b_1\}\times\ldots\times\{\tup a_k,\tup b_k\}\smallsetminus\{(\tup b_1,\ldots,\tup b_k)\} \\
&\Downarrow \\
p(\tup b_1,\ldots,\tup b_k,\tup u) &= p(\tup b_1,\ldots,\tup b_k,\tup v). 
\end{align*}
The \emph{class of supernilpotence}, $\cl_{sn}(A)$, is the smallest $k$ such that $A$ is $k$-supernilpotent.
\end{definition}

Note that 1-nilpotence and 1-supernilpotence are exactly the same conditions. Such algebras are called \emph{abelian}.

A $k$-ary operation $p$ is called \emph{absorbing at $a_1,\dots,a_k$ into $e$} if $p(u_1,\dots,u_k)=e$ whenever there is $i$ such that $u_i=a_i$.
More generally, an operation $p$ is called \emph{$k$-batch-absorbing at $\tup a_1,\dots,\tup a_k$ into $e$} if $p(\tup u_1,\dots,\tup u_k)=e$ whenever there is $i$ such that $\tup u_i=\tup a_i$.

\begin{lemma}\label{l:abs}
Let $A$ be a $k$-supernilpotent algebra. Then, for every $l>k$, all $l$-batch-absorbing polynomial operations are constant.
In particular, all absorbing polynomial operations of arity $>k$ are constant.
\end{lemma}

\begin{proof}
Let $p$ be a $(k+1)$-batch-absorbing polynomial at $\tup a_1,\dots,\tup a_{k+1}$ into $e$. Let $\tup b_1,\dots,\tup b_{k+1}$ be tuples of elements of $A$. We will prove that $p(\tup b_1,\dots,\tup b_{k+1})=e$.
Let $\tup u=\tup a_{k+1}$ and $\tup v=\tup b_{k+1}$. Indeed, $p$ satisfies the assumptions of the implication from Definition \ref{d:supnilp}, since $\tup x_i=\tup a_i$ for at least one $i$. The conclusion then reads $p(\tup b_1,\ldots,\tup b_k, \tup b_{k+1})=p(\tup b_1,\ldots, \tup b_k, \tup a_{k+1})=e$ by absorption.

Let $p$ be an $l$-batch-absorbing polynomial at $\tup a_1,\dots,\tup a_{l}$, with $l>k+1$. Let $\tup b_1,\ldots,\tup b_l$ be tuples of elements of $A$. To prove that $p(\tup b_1,\dots,\tup b_l)=e$, consider the polynomial $q(\tup x_1,\ldots,\tup x_{k+1})=p(\tup x_1,\ldots,\tup x_{k+1},\tup b_{k+2},\ldots,\tup b_l)$, observe that it is $(k+1)$-batch-absorbing at $\tup a_1,\dots,\tup a_{k+1}$, and thus constant onto $e$.
\end{proof}

It follows from the second part of the proof that if all absorbing polynomials of arity $k+1$ are constant, then this holds for every arity $>k$, and similarly for batch-absorption.


For Mal'tsev algebras, the converse implication in Lemma~\ref{l:abs} also holds, although the proof is not so simple. Also, for Mal'tsev algebras, we can replace tuples by single elements in Definition \ref{d:supnilp}. Theorem \ref{t:supnilp} collects several equivalent definitions of supernilpotence in Mal'tsev algebras. Condition (4) provides an algorithm to check $k$-supernilpotence. Condition (4) implies that $k$-supernilpotent algebras in a variety $\mathcal V$ form a subvariety. 

To state condition (3), we need the following technical definitions.
A \emph{fork} in a relation $R\subseteq A^n$ is a pair $(\tup u,\tup v)\in R^2$ satisfying $u_n\neq v_n$ and $u_i=v_i$ for all $i<n$.
For a positive integer $k$, let $k(i)$ denote the $i$-th digit from the right of the binary expansion of $k$. For a pair $(a,b)$, denote $(a,b)_{(0)}=a$, $(a,b)_{(1)}=b$ and define
$c_i^n(a, b) = ((a,b)_{(k(i))}:k=0,\ldots,2^n-1)\in A^{2^n}$.

To state condition (4), we need a sequence of terms $q_n$ in $2^n-1$ variables (using a Mal'tsev term $m$ as a parameter) defined inductivily as follows: 
\[ q_2=m(y,x,z), \qquad q_{n+1}=m(x_{2^n},q_n(x_1,\ldots,x_{2^n-1}),q_n(x_{2^n+1},\ldots,x_{2^{n+1}-1})).\]

\begin{theorem}\cite{AM,AMO,Op}\label{t:supnilp}
Let $A$ be a Mal'tsev algebra. The following conditions are equivalent:
\begin{enumerate}
	\item $A$ is $k$-supernilpotent.
	\item[(1*)] For every $(k+1)$-ary polynomial operation $p$ and all $a_1\neq b_1$, \dots, $a_k\neq b_k$ and $u,v$ from $A$, 
\begin{align*}
p(x_1,\ldots,x_k, u) &= p(x_1,\ldots, x_k, v) \qquad \forall ( x_1,\ldots, x_k)\in\{ a_1, b_1\}\times\ldots\times\{ a_k, b_k\}\smallsetminus\{( b_1,\ldots, b_k)\} \\
&\Downarrow \\
p( b_1,\ldots, b_k, u) &= p( b_1,\ldots, b_k, v).
\end{align*}
	\item Every $l$-batch-absorbing polynomial operation, $l>k$, is constant.
	\item[(2*)] Every absorbing polynomial operation of arity $>k$ is constant.
	\item The subalgebra of $A^{2^{k+1}}$ generated by $\{ c_i^{k+1}(a,b):\,  a,b\in A,\, i=1,\ldots,k+1 \}$ contains no fork.
	\item For every term $t$ and all pairs of tuples $\tup a_1,\tup b_1,\ldots,\tup a_{k+1},\tup b_{k+1}$ of elements of $A$, 
	\[ q_{k+1}(t(\tup a_1,\ldots,\tup a_{k+1}),t(\tup a_1,\ldots,\tup a_k,\tup b_{k+1}),\ldots,t(\tup b_1,\ldots,\tup b_k,\tup a_{k+1}))=t(\tup b_1,\ldots,\tup b_{k+1}) \]
	(the parameters of $t$ on the left hand side are all combinations of $\tup a$'s and $\tup b$'s except the one on the right hand side).
\end{enumerate}
\end{theorem}

Equivalence of conditions (1),(1*),(2*) is proved in \cite{AM}, equivalence of condition (2) follows from Lemma \ref{l:abs}, condition (3) appears in \cite{Op}, condition (4) appears in \cite{AMO}.


Finite supernilpotent Mal'tsev algebras admit prime decomposition. This result can be found in \cite[Section 7]{AM} although a major part of the proof is based on results of Kearnes \cite{Kea}.

\begin{theorem}\label{t:supnilp_fin}\cite{AM}
Let $A$ be a finite Mal'tsev algebra. Then $A$ is $k$-supernilpotent for some $k$ if and only if $A$ is a direct product of nilpotent algebras of prime power size.
\end{theorem}

\section{(Super)nilpotence in loops}\label{sec:loops}

\subsection{Loops}

Let $(Q,\cdot,\ldiv,\rdiv,1)$ be a loop, i.e., $\cdot$ is a binary operation, 1 is its unit element, $a\ldiv b$ is the unique solution to the equation $a*x=b$ and dually for $/$. We denote $L_x(y)=xy$, $R_x(y)=yx$ the \emph{left} and \emph{right translations}. The \emph{multiplication group} of $Q$ is the permutation group generated by the translations, $\mlt{Q}=\langle L_x,R_x:\ x\in Q\rangle$. The stabilizer of the unit element is called the \emph{inner mapping group}, $\inn Q=\mlt Q_1$. The inner mapping group is generated by the mappings \[L_{x,y} = L_{xy}^{-1}L_xL_y,\quad     R_{x,y} = R_{yx}^{-1}R_xR_y,\quad     T_x =R_x^{-1}L_x.\]
A subloop of $Q$ invariant with respect to the action of $\inn Q$ is called \emph{normal}. Normal subloops are precisely the kernels of homomorphisms. If $N$ is a normal subloop of $Q$, then $|N|$ divides $|Q|$. If $M,N$ are normal subloops of $Q$ such that $M\cap N=1$, then $|MN|=|M|\cdot|N|$.

For a general background in loops, we refer to the textbook \cite{Pfl}. 

\subsection{Nilpotence}

The theory of nilpotent loops has been developed since the very beginnings of loop theory, with definitions taken in direct analogy to group theory (for details, see the survey article \cite{NR}).

Fix any commutator term $[x,y]$, i.e., a term such that $[x,y]=1$ if and only if $xy=yx$. Fix any associator term $[x,y,z]$, i.e., a term such that $[x,y,z]=1$ if and only if $x(yz)=(xy)z$. The following definition is a straightforward generalization from group theory.

\begin{definition}
The \emph{center of a loop} $Q$ is defined by \[ Z(Q)=\{a\in Q:\ [a,x]=[a,x,y]=[x,a,y]=[x,y,a]=1 \text{ for all } x,y\in Q\}. \]
A loop $Q$ is called \emph{$k$-centrally-nilpotent} if there is a series of normal subloops $1=N_0\leq N_1\leq\ldots\leq N_k=Q$ such that $N_{i+1}/N_i\leq Z(Q/N_i)$.
The \emph{class of central nilpotence}, $\cl_{cn}(Q)$, is the smallest $k$ such that $A$ is $k$-centrally-nilpotent.
\end{definition}

Note that $a\in Z(Q)$ if and only if $f(a)=a$ for every $f\in\inn Q$.

We define the \emph{upper central series} of $Q$ as the series of normal subloops $Z_i(Q)$, $i=0,1,\dots$ such that $Z_0(Q)=1$ and $Z_{i+1}(Q)/Z_i(Q)=Z(Q/Z_i(Q))$. Clearly, if $Z_k(Q)=Q$, then $Q$ is $k$-centrally-nilpotent, as witnessed by the series of iterated centers. The converse implication is also true, and not difficult to prove.

A loop is centrally nilpotent if and only if it is nilpotent in the sense of Section \ref{sec:ua}, see \cite[Section 10]{SV1} for a proof and \cite{SV2} for a further discussion.

As noted earlier, loops of prime power order need not be nilpotent (since $|Z(Q)|$ divides $|Q|$, any non-associative loop of prime order fails to be nilpotent) and nilpotent loops generally do not admit a prime decomposition.

\begin{example}\label{ex:6}
The following multiplication table shows a loop $Q$ of order 6 such that $\cl_{cn}(Q)=2$, but it is directly indecomposable.
\begin{displaymath}
\begin{array}{r|rrrrrr}
    Q&1&2&3&4&5&6\\
    \hline
    1&1&2&3&4&5&6\\
    2&2&1&4&3&6&5\\
    3&3&4&5&6&1&2\\
    4&4&3&6&5&2&1\\
    5&5&6&2&1&3&4\\
    6&6&5&1&2&4&3
\end{array}
\end{displaymath}
The center is $Z(Q)=\{1,2\}$, the factor $Q/Z(Q)$ is a cyclic group of order 3, hence $\cl_{cn}(Q)=2$. It is the only proper normal subloop of $Q$, hence direct indecomposability. 
\end{example}

\subsection{Nilpotence of the multiplication group}

The \emph{multiplication group} $\mlt Q$ of a loop $Q$ is the subgroup of the symmetric group $S_Q$ generated by all left and right translations (i.e., the mappings $x\mapsto ax$ and $x\mapsto xa$, for every $a\in Q$).
The loop $Q$ from Example \ref{ex:6} is 2-centrally-nilpotent, however, $\mlt Q$ is a non-nilpotent group of order 24.

Bruck \cite{Br} proved that loops with a nilpotent multiplication group are centrally nilpotent, and Wright \cite{Wr} proved that finite loops with a nilpotent multiplication group admit a prime decomposition. We present both proofs in Section \ref{sec:proofs}. 



\subsection{Supernilpotence}

Loops are Mal'tsev algebras, therefore, we can define supernilpotence by any of the equivalent conditions of Theorem \ref{t:supnilp}. In our opinion, the absorption conditions (2) and (2*) are the most natural ones. We will use them in our proof of Theorem \ref{thm:main}, and also in the subsequent paper \cite{SV-sn}.

We start with an observation that in loops, without loss of generality, we can consider only absorbing polynomials at $(1,\ldots,1)$ into 1. 

\begin{observation}
Let $Q$ be a loop. Then $Q$ is $k$-supernilpotent if and only if all polynomials of arity $>k$ absorbing at $(1,\ldots,1)$ into 1 are constant.
\end{observation}

\begin{proof}
$(\Rightarrow)$ follows from Lemma \ref{l:abs}.
$(\Leftarrow)$ Let $p(x_1,\ldots,x_n)$ be a polynomial absorbing at $\tup a$ into $e$. Let $q(x_1,\ldots,x_n)=p(x_1a_1,\ldots,x_na_n)/e$. Then $q$ is absorbing at $(1,\ldots,1)$ into 1, thus constant. 
Consequently, $p(x_1,\ldots,x_n)=q(x_1/a_1,\ldots,x_n/a_n)e$ is constant onto $e$.
\end{proof}

In the rest of the paper, all absorbing polynomials are implicitly understood to be absorbing $(1,\ldots,1)$ into 1. A similar statement applies to batch-absorption.

Note that any commutator or associator term is absorbing. Therefore, 2-supernilpotent loops are associative and 1-supernilpotent loops are also commutative. Since in groups the classes of nilpotence and supernilpotence coincide, we obtain that
\begin{itemize}
	\item a loop is 1-supernilpotent if and only if it is an abelian group,
 \item a loop is 2-supernilpotent if and only if it is a 2-nilpotent group.
\end{itemize}

For central nilpotence, we have $\cl_{cn}(Q/Z(Q))=\cl_{cn}(Q)-1$. For supernilpotence, the difference may be greater than 1. We will show that it is at least 1.

\begin{proposition}\label{p:Q/Z}
Let $Q$ be a supernilpotent loop. Then $\cl_{sn}(Q/Z(Q))<\cl_{sn}(Q)$.
\end{proposition}

\begin{proof}
Let $p$ be a $k$-ary absorbing polynomial on $Q/Z(Q)$. Let $q$ result from $p$ by replacing every constant $aZ(Q)$ by $a$. Then $q$ is a polynomial on $Q$ such that
$q(a_1,\ldots,a_k)\in Z(Q)$ whenever $a_i\in Z(Q)$ for some $i$. Let
\begin{align*}
r(y,x_1,\ldots,x_k)&=[y,q(x_1,\ldots,x_k)],\\
 r_1(y,z,x_1,\ldots,x_k)&=[y,z,q(x_1,\ldots,x_k)],\\ r_2(y,z,x_1,\ldots,x_k)&=[y,q(x_1,\ldots,x_k),z],\\ r_3(y,z,x_1,\ldots,x_k)&=[q(x_1,\ldots,x_k),y,z].
\end{align*}
Then $r,r_1,r_2,r_3$ are absorbing polynomials on $Q$ of arity $>k$, hence constant on $Q$. Consequently, $q(a_1,\ldots,a_k)\in Z(Q)$ for all $a_1,\ldots,a_k\in Q$, and thus
$p(a_1Z(Q),\ldots,a_kZ(Q))=1Z(Q)$ in $Q/Z(Q)$.
\end{proof}

(This proposition actually holds for every Mal'tsev algebra $A$: if $A$ is $k$-supernilpotent, then, using inequality (HC8) of \cite{AM}, $[[1_A,\dots,1_A]_k,1_A]\leq [1_A,\dots,1_A]_{k+1}=0_A$, 
which implies that $[1,\dots,1]_k$ is contained in the center $\zeta_A$ of $A$, and so $A/\zeta_A$ is $(k-1)$-supernilpotent. However note that the loop theoretic proof is rather elementary, not relying on the inequality (HC8).)


Condition (3) of Theorem \ref{t:supnilp} suggests an algorithmic procedure for checking supernilpotence. 
Subalgebra generation is efficient with respect to the subalgebra size, and searching for a fork is a simple task.
Nevertheless, the subalgebra of $A^{2^k}$ described in the condition may be (and often is) very large, even for small $k$ and $|A|$. 
Storing the elements of the power in a tree data structure allows to identify forks instantly. With some optimization, we were able to find forks for certain loops of size $8$, thus proving that they are not 3-supernilpotent. 

\begin{example}\label{ex:8}
There are 134 non-associative nilpotent loops of order 8, listed in the {\tt loops} package for the computer system {\tt GAP} \cite{loopspackage}. All of them have $\cl_{cn}=2$ and $\cl_{m}\in\{3,4\}$.

There are 62 nilpotent loops of order 8 with $\cl_m=3$. For 34 of them, we calculated that $\cl_{sn}>3$. We expect that the remaining 28 loops have $\cl_{sn}=3$, but we have no tool to prove so.
\end{example}

A detailed report on the computational experiments, including a description of the algorithm, can be found in \cite{Z-diplomka}.

\section{Proof of Theorem \ref{thm:main}}\label{sec:proofs}

\subsection{Proof of $\cl_m(Q)\leq\cl_{sn}(Q)$}

First, let us formally introduce a correspondence between words in the multiplication group and terms of the loop.

Fix a set of variables $X$. 
An \emph{m-word} of length $l$ is a formal expression of the form $(U^{(1)}_{x_1})^{k_1}\ldots(U^{(l)}_{x_l})^{k_l}$ where $U^{(i)}\in\{L,R\}$, $x_1,\ldots,x_l\in X$ and $k_1,\ldots,k_l\in\{\pm1\}$.
Expressions using group commutators will be also understood as m-words, for example, the expression $[L_x,R_y^{-1}]$ shall be understood as the m-word $L_x^{-1}R_yL_xR_y^{-1}$.

An m-word $W$ containing variables $x_1,\ldots,x_n\in X$ can be converted naturally into a loop term $t_W(x_1,\ldots,x_n,z)$. For example, for $W=L_yR_x^{-1}$ we have $t_W(x,y,z)=y(z/x)$, interpreting the expression $L_yR_x^{-1}(z)$ as a term.
Formally, we define $t_W$ recursively: 
let $t_{l+1}=z$; for $i=l,\ldots,1$ we define $t_{i}=x_{i}t_{i+1}$ if $U^{(i)}=L$ and $k_i=1$, $t_{i}=x_{i}\ldiv t_{i+1}$ if $U^{(i)}=L$ and $k_i=-1$, $t_{i}=t_{i+1}x_{i}$ if $U^{(i)}=R$ and $k_i=1$, $t_{i}=t_{i+1}/x_{i}$ if $U^{(i)}=R$ and $k_i=-1$; finally, let $t_W=t_1$.

Observe that for every $f\in\mlt Q$ there is an m-word $W$ and $a_1,\ldots,a_n\in Q$ such that $f(q)=t_W(a_1,\ldots,a_n,q)$ for every $q\in Q$. 



\begin{proof}
Assume that $Q$ is $k$-supernilpotent. To prove that $\mlt Q$ is $k$-nilpotent, we will show that
\[ [f_1,[\ldots,[f_{k},f_{k+1}]]]=id \]
for all $f_1,\ldots,f_{k+1}\in\mlt Q$.
This is equivalent to proving that $Q$ satisfies every identity 
\[ t_W(\tup x_1,\ldots,\tup x_{k+1},z)=z \] 
where 
\[ W=[W_1,[\ldots,[W_k,W_{k+1}]]] \]
and $W_1(\tup x_1), \ldots, W_{k+1}(\tup x_{k+1})$ are arbitrary m-words. 
Note that $t_W(\tup u_1,\ldots,\tup u_{k+1},q)=q$ whenever there is $i$ such that $\tup u_i=(1,\ldots,1)$. Therefore, 
the polynomial \[ p(\tup x_1,\ldots,\tup x_{k+1})=t_W(\tup x_1,\ldots,\tup x_{k+1},q)/q \] is $(k+1)$-batch-absorbing, thus constant onto 1, and
$t_W(\tup u_1,\ldots,\tup u_{k+1},q)=q$ for every $\tup u_1,\dots,\tup u_{k+1}$ and $q$.
%
\end{proof}

\subsection{Proof of $\cl_{cn}(Q)\leq\cl_{m}(Q)$}

Let $Q$ be a loop. We define a series $\mathcal N_0 \leq \mathcal N_1 \leq \ldots$ of subgroups of $\mlt Q$ inductively: let $\mathcal N_0=\inn Q$ and let $\mathcal N_{i+1}$ be the normalizer of $\mathcal N_i$ in $\mlt Q$. 

\begin{lemma}\label{l:ineq2}
Let $Q$ be a loop and $\mathcal Z_0\leq\mathcal Z_1\leq\ldots$ the upper central series of $\mlt Q$. Then $\mathcal N_i \geq \mathcal Z_i$ for every $i$. 
\end{lemma}

\begin{proof}
We proceed by induction on $i$. For $i=0$, we have $\inn Q=\mathcal N_0\geq\mathcal Z_0=1$. 

Suppose that $\mathcal N_i \geq \mathcal Z_i$ and let $f \in \mathcal Z_{i+1}$. By definition, $f\mathcal Z_i \in Z(\mlt Q/\mathcal Z_i)$. 
For every $g \in \mathcal N_{i}$, $fgf^{-1}\mathcal Z_i=g \mathcal Z_i$, and thus, by the induction assumption,
$fgf^{-1}\mathcal N_i=g \mathcal N_i =\mathcal N_i$, that is, $fgf^{-1} \in \mathcal N_i$. 
It follows that $f$ is in the normalizer of $\mathcal N_i$, which is $\mathcal N_{i+1}$.
\end{proof}

Observe that, for every $t\in\mlt Q$, we have $t=R_{t(1)}h$ for some $h\in\inn Q$.

\begin{lemma} 
\label{p:normalizer}
Let $Q$ be a loop and $Z_0\leq Z_1\leq\ldots$ its upper central series. 
Then, for every $i$,
\[\mathcal N_i=\{R_af :\ a \in Z_i, f \in \inn Q \}.\]
In particular, $Q$ is $k$-nilpotent if and only if $\mathcal N_k=\mlt Q$. 
\end{lemma}


\begin{proof}
We prove the claim by induction on $i$. The case $i=0$ is obvious, so suppose that the claim holds for some $i$ and we prove it for $i+1$.

$(\subseteq)$ Let $t=R_{t(1)}h \in N_{i+1}$. It is sufficient to show that $t(1)\in Z_{i+1}$. That is, $t(1)Z_i=Z(Q/Z_i)$, which means $f(t(1))Z_i=t(1)Z_i$ for every $f\in\inn Q$.

Let $f\in\inn Q$. Since $t$ is in the normalizer of $\mathcal N_i$, for every $a\in Z_i$ we have $(R_af)t=t(R_bg)$ for some $b \in Z_i$, $g \in \inn Q$. Applying the equality on the unit element, we obtain
\[ f(t(1))a = t(g(1)b) = t(b) = h(b)t(1), \]
and thus $f(t(1))Z_i=Z_it(1)=t(1)Z_i$ which we wanted to prove.

$(\supseteq)$ Let $t=R_xh$, $h\in \inn Q$, $x \in Z_{i+1}$, we prove that it normalizes every element of $\mathcal N_i$. By the induction assumption, consider the element $R_af$ where $a \in Z_i$, $f \in \inn Q$. Then $R_aft=R_yh_1$, where $h_1 \in \inn Q$ and $y=R_aft(1)=f(x)\cdot a$. Since $x\in Z_{i+1}$, we have $f(x)Z_i=xZ_i=Z_ix$. Therefore, $y=f(x)\cdot a = cx$ for some $c \in Z_i$.
Now, let $h_2=R_c^{-1}R_x^{-1}R_{cx}h_1 \in \inn Q$ and calculate
\[R_aft=R_{cx}h_1=R_xR_ch_2=(R_xh)(h^{-1}R_ch_2)=(R_xh)(R_bg)=t(R_bg),\]
for $b=h^{-1}(c) \in Z_i$ and some $g \in \inn Q$. It follows that $t\in \mathcal N_{i+1}$.

The last claim follows from the fact that $\{R_a:a\in Q\}$ is a transversal of $\inn Q$ in $\mlt Q$.
\end{proof}

The inequality $\cl_{cn}(Q)\leq\cl_{m}(Q)$ immediately follows, since, by Lemma \ref{l:ineq2}, the series $\mathcal N_i$ terminates not later than the series $\mathcal Z_i$.

\subsection{Proof of $\cl_{m}(Q)<\infty\Rightarrow\cl_{sn}(Q)<\infty$}

For every loop $Q$, there is a (covariant) Galois correspondence (cf. \cite[Section 2.5]{Berg}) between normal subloops of $Q$ and normal subgroups of $\mlt Q$, given by
\begin{align*}
\mathrm{NSub}(Q) &\leftrightarrow \mathrm{NSub}(\mlt Q)\\
N &\rightarrowtail N^*=\{f\in\mlt Q: \ f(x)/x\in N\text{ for all }x\in Q\} \\
G(1)=\{g(1):g\in G\} &\leftarrowtail G
\end{align*}
Checking all the properties is routine. In our proof, we will use the following two of them:
\begin{itemize}
	\item[(G1)] \emph{$G(1)$ is a normal subloop of $Q$}, since, for every $f\in \inn Q$ and $g\in G$, we have $fgf^{-1}\in G$ and thus $f(g(1))=fgf^{-1}(f(1))=fgf^{-1}(1)\in G(1)$;
	\item[(G2)] $G\subseteq G(1)^*$, since, for every $g\in G$ and $x \in Q$, we have $R_x^{-1}gR_x\in G$ and thus $g(x)/x=R_x^{-1}gR_x(1)\in G(1)$. 
\end{itemize}
Furthermore, observe that 
\begin{enumerate}
	\item[(G3)] \emph{$|G(1)|$ divides $|G|$}, since $|G(1)|=[G:G_1]$ (the orbit size equals the index of the stabilizer);
  \item[(G4)] \emph{if $N$ is a normal subloop of $Q$, then $\mlt{Q/N}\simeq\mlt Q/N^*$}: consider the homomorphism $\mlt Q\to\mlt{Q/N}$, $f\mapsto(xN\mapsto f(x)N)$, calculate its kernel $\{f:f(x)N=xN$ for all $x\in Q\}=N^*$ and use the first isomorphism theorem.
\end{enumerate}

\begin{proposition}[{\cite[Lemma 2.2 of Section VI.2]{Br-book}}]\label{p:pgroup}
Let $Q$ be a finite loop and $p$ a prime. Then $Q$ is nilpotent of order $p^k$ if and only if $\mlt Q$ is a $p$-group.
\end{proposition}

\begin{proof}
$(\Leftarrow)$ Since $p$-groups are nilpotent, $Q$ is nilpotent by the result of the previous section. If $m$ divides $|Q|$, then $m$ also divides $|\mlt Q|=|Q|\cdot|\inn Q|$, hence the only prime divisor of $|Q|$ is $p$.

$(\Rightarrow)$ We will proceed by induction on $\cl_{cn}(Q)$. If $Q$ is an abelian group, then $\mlt Q\simeq Q$ and we are done. In the induction step, consider the factor $Q/Z(Q)$. Its nilpotence class is smaller, hence $\mlt{Q/Z(Q)}$ is a $p$-group. Using (G4), we obtain $|\mlt Q|=|\mlt{Q/Z(Q)}|\cdot|Z(Q)^*|$, hence it remains to show that $Z(Q)^*$ is a $p$-group. 

For $f$ in $Z(Q)^*$, consider the mapping $\varphi_f(x)=f(x)/x$. By definition of $Z(Q)^*$, the mapping $\varphi_f$ maps $Q$ into $Z(Q)$, thus it belongs to the direct power $Z(Q)^Q$. 
The mapping $Z(Q)^*\to Z(Q)^Q$, $f\mapsto\varphi_f$ is clearly injective, and we show that it is a group homomorphism: for every $x\in Q$,
\[ \varphi_{fg}(x)=f(g(x))/x=f((g(x)/x)\cdot x)/x=(f(x)/x)(g(x)/x)=\varphi_f(x)\varphi_g(x),\] 
where the crucial equality follows from the observation that $h(zx)=zh(x)=h(x)z$ for every $x\in Q$, $z\in Z(Q)$ and $h\in\mlt Q$.
Consequently, $Z(Q)^*$ is isomorphic to a subgroup of $Z(Q)^Q$, which is an abelian $p$-group.
\end{proof}

\begin{proof}[Proof of $\cl_{m}(Q)<\infty\Rightarrow\cl_{sn}(Q)<\infty$]
Let $Q$ be a finite loop such that $\mlt Q$ is nilpotent. In particular, $Q$ is also nilpotent, as we proved earlier, and so are all subloops of $Q$. If we prove that $Q$ admits a prime decomposition, then supernilpotence follows from Theorem \ref{t:supnilp_fin}.

We find the prime decomposition by induction on $|Q|$. If $Q$ is trivial or $|Q|$ is a prime power, it admits trivial prime decomposition. So assume that $|Q|=p^er$ where $p$ is prime, $e\geq1$, $p\nmid r\neq1$. 

Under the assumptions, $\mlt Q$ is not a $p$-group, otherwise violating Proposition \ref{p:pgroup}. 
Let $\mlt Q=PR$ be an internal direct product where $P$ is the Sylow $p$-subgroup (it exists, since $\mlt Q$ is nilpotent). 
By (G1), both $P(1)$ and $R(1)$ are normal subloops of $Q$, and (G3) implies that $|P(1)|$ is a power of $p$, and $p$ does not divide $|R(1)|$. 
By the Lagrange property of normal subloops, $P(1)\cap R(1)=1$.
We will show that $Q=P(1)R(1)$. 

By (G2), $\mlt Q=PR\subseteq P(1)^*R(1)^*$, hence $P(1)^*R(1)^*=\mlt Q$. We also have $P(1)^* \cap R(1)^* = 1$, since, if $f \in P(1)^* \cap R(1)^*$, then $f(x)/x \in P(1)\cap R(1)=1$ for every $x\in Q$. 
Therefore, $\mlt Q=P(1)^*R(1)^*$ is also an internal direct decomposition, and comparing the numbers of elements, we obtain that $P(1)^*= P$ and $R(1)^*= R$.
Using (G4), \[\mlt{Q/R(1)}\simeq\mlt Q/R(1)^*=\mlt Q/R\simeq P, \] which is a $p$-group, hence $|Q/R(1)|$ is a power of $p$ by Proposition \ref{p:pgroup}, and thus $r$ divides $R(1)$.
Using (G4) again, \[\mlt{Q/P(1)}\simeq\mlt Q/P(1)^*=\mlt Q/P\simeq R, \] a nilpotent group, hence we can apply the induction assumption, obtain a direct decomposition $Q/P(1)\simeq \prod Q_{p_i}$, hence also a direct decomposition $R\simeq\mlt{Q/P(1)}\simeq\prod\mlt{Q_{p_i}}$, and we see from Proposition \ref{p:pgroup} that none of $p_i$ equals $p$. Therefore, $p$ does not divide $Q/P(1)$, and thus $|P(1)|=p^e$ (we already know that $|P(1)|$ is a power of $p$).
Consequently, $|P(1)R(1)|=|P(1)|\cdot|R(1)|\geq p^er=|Q|$ (the first equality follows from $P(1)\cap R(1)=1$), and thus $Q=P(1)R(1)$.

Consequently, $Q\simeq P(1)\times R(1)$ is a direct decomposition, $|P(1)|=p^e$, $|R(1)|=r<|Q|$. Applying the induction assumption on $R(1)\simeq Q/P(1)$, we obtain a prime decomposition of $Q$.
\end{proof}

\section{Open problems}

\begin{problem}
Does the equivalence \[ \cl_{sn}(Q)<\infty\quad\Leftrightarrow\quad \cl_{m}(Q)<\infty\] hold for every loop $Q$? Stated differently, is every loop $Q$ with nilpotent multiplication group supernilpotent?
\end{problem}

\begin{problem}\
\begin{enumerate}
	\item Find a function $f$ such that $\cl_{sn}(Q)\leq f(\cl_{cn}(Q))$ for every supernilpotent loop $Q$, or prove that no such function exists.
	\item Find a function $g$ such that $\cl_{sn}(Q)\leq g(\cl_m(Q))$ for every supernilpotent loop $Q$, or prove that no such function exists.
\end{enumerate}
Examples of loops $Q$ with small $\cl_{cn}(Q)$ or $\cl_{m}(Q)$ and large $\cl_{sn}(Q)$ are also of interest, since they provide lower bounds for $f,g$.
\end{problem}


\begin{thebibliography}{99}

\bibitem{AE}
E. Aichinger, J. Ecker, \emph{Every $(k+1)$-affine complete nilpotent group of class $k$ is affine complete}. Internat. J. Algebra Comput. 16 (2006), no. 2, 259--274.

\bibitem{AM}
E.~Aichinger and N.~Mudrinski, \emph{Some applications of higher commutators in Mal'cev algebras},
Algebra Universalis \textbf{63} (2010), no. \textbf{4}  , 367--403.

\bibitem{AMO}
E.~Aichinger, N.~Mudrinski, J. Opr\v sal, \emph{Complexity of term representations of finitary functions.} Internat. J. Algebra Comput. 28 (2018), no. 6, 1101--1118.

\bibitem{Berg}
C.~Bergman, \emph{Universal algebra: Fundamentals and selected topics}, Chapman \& Hall/CRC Press, 2011.

\bibitem{Br}
R.H. Bruck, \emph{Contributions to the theory of loops}, Trans. Amer. Math. Soc. \textbf{60} (1946), 245--354.

\bibitem{Br-book}
R.H.~Bruck, \emph{A survey of binary systems}, Ergebnisse der Mathematik und ihrer Grenzgebiete, Springer Verlag, 3rd edition, 1971.

\bibitem{Bu}
A. Bulatov, \emph{On the number of finite Mal’tsev algebras}, Contributions to General Algebra \textbf{13} (2001), 41--54.



\bibitem{FM}
R.~Freese and R.~McKenzie, \emph{Commutator theory for congruence modular varieties}, London Mathematical Society Lecture Note Series \textbf{125}, Cambridge University Press, Cambridge, 1987.


\bibitem{G2}
G. Glauberman, \emph{On loops of odd order. II.} J. Algebra 8 (1968), 393--414.

\bibitem{GW}
G. Glauberman, C. R. B. Wright, \emph{Nilpotence of finite Moufang 2-loops.} J. Algebra 8 (1968), 415--417.

\bibitem{HH}
J. Hagemann, C. Herrmann. \emph{A concrete ideal multiplication for algebraic systems and its relations to congruence distributivity.} Arch. Math. (Basel), 32 (1979),234--245.

\bibitem{Kea}
K. A. Kearnes, \emph{Congruence modular varieties with small free spectra.} Algebra Universalis 42 (1999), 165--181.


\bibitem{Moo}
A. Moorhead, \emph{Higher commutator theory for congruence modular varieties.} J. Algebra 513 (2018), 133--158. 

\bibitem{loopspackage}
G. P.~Nagy and P.~Vojt\v{e}chovsk\'y, \texttt{LOOPS}: Computing with quasigroups and loops in GAP, version 2.2.0, available at \texttt{www.math.du.edu/loops}.


\bibitem{NR}
M.~Niemenmaa, M.~Rytty, \emph{Centrally nilpotent finite loops}, Quasigroups Related Systems \textbf{19} (2011), no. \textbf{1}, 123–-132.

\bibitem{Op}
J. Opr\v sal, \emph{A relational description of higher commutators in Mal'cev varieties.} Algebra Universalis 76 (2016), no. 3, 367--383. 

\bibitem{Pfl}
H. O. Pflugfelder, \emph{Quasigroups and loops: Introduction}, Heldermann Verlag, Berlin, 1990.


\bibitem{Z-diplomka}
\v Z. Semani\v sinov\'a. \emph{Higher commutators in loop theory.} Master Thesis, Charles University, Prague, 2021.

\bibitem{Sm}
J. D. H. Smith. \emph{Mal’cev varieties}, Lecture Notes in Math. 554, Springer Verlag Berlin, 1976.


\bibitem{SV1}
D. Stanovsk\'y and P. Vojt\v{e}chovsk\'y, \emph{Commutator theory for loops}, J. Algebra \textbf{399} (2014), 290--322.

\bibitem{SV2}
D. Stanovsk\'y and P. Vojt\v{e}chovsk\'y, \emph{Abelian extensions and solvable loops}, Results in Math. 66/3-4 (2014), 367--384.

\bibitem{SV-sn}
D. Stanovsk\'y and P. Vojt\v{e}chovsk\'y, \emph{Supernilpotent groups and 3-supernilpotent loops}, to appear in J. Algebra Appl.


\bibitem{Wr}
C. R. B.~Wright, \emph{On the multiplication group of a loop}, Illinois J. Math. \textbf{13} (1969), 660--673.

\end{thebibliography}
\end{document}